\documentclass[a4paper,12pt, oneside]{amsart}

\usepackage{amsthm, amsmath}
\usepackage{amssymb,amsfonts,amscd,verbatim, longtable}
\usepackage{xy}
\usepackage{mathrsfs}
\usepackage{wrapfig}
\xyoption{all}
\SelectTips{cm}{10}

\usepackage[hypertex]{hyperref}
\newtheorem{tr}{Theorem}[section]
\newtheorem*{tr*}{Theorem}
\newtheorem{lemma}[tr]{Lemma}
\newtheorem{pr}[tr]{Proposition}

\newtheorem{cor}[tr]{Corollary}%[section]
\newtheorem{df}[tr]{Definition}%[section]

\newtheorem{rem}[tr]{Remark}

\def\differential{d}
\renewcommand\d\differential

\DeclareMathOperator\im{Im}

\DeclareMathOperator\Hom{Hom}

\DeclareMathOperator\Aut{Aut}

\DeclareMathOperator\Ext{Ext}

\DeclareMathOperator\id{id}
\DeclareMathOperator\GL{GL}

\DeclareMathOperator\Bl{Bl}
\DeclareMathOperator\Sing{Sing}

\def\k{\Bbbk}
\renewcommand\Im\im
\def\bb#1{\mathbb #1}
\def\cal#1{\mathcal #1}
\def\fcal#1{\boldsymbol{\cal{#1}}}

\def\dd#1#2{\frac{\partial #1}{\partial#2} }

\def\ra{\rightarrow}
\def\xra{\xrightarrow}

\def\mto{\mapsto}

\def\pmat#1{\begin{pmatrix}#1\end{pmatrix}}

\def\smat#1{\left(\begin{smallmatrix}#1\end{smallmatrix}\right)}

\def\point#1{\langle #1 \rangle}

\def\refeq#1{$(\ref{#1})$}

\def\cExt{{\mathscr E\!\mathit{xt} } }

\def\P{\bb P}

\def\O{\cal O}

\def\ten{\otimes}
\def\sqten{\boxtimes}
\let\le\leq
\let\ge\geq
\let\star *
\let\subset\subseteq

\def\defeq{:=}

\def\iso{\cong}

\def\tilde{\widetilde}

\title[Universal plane curve and moduli spaces of $1$-dimensional coherent sheaves]{Universal plane curve and moduli spaces of $1$-dimensional coherent sheaves}
\author{Oleksandr Iena}
\address{University of Luxembourg, Campus Kirchberg\\
Mathematics Research Unit\\
6, rue Richard Coudenhove-Kalergi\\
L-1359 Luxembourg City\\
Grand Duchy of Luxembourg}
\email{oleksandr.iena@uni.lu}
\date{}

\begin{document}
\begin{abstract}
We show that the universal plane curve $M$ of fixed degree $d\ge 3$ can be seen as
a closed subvariety in a certain Simpson moduli space of $1$-dimensional sheaves on $\P_2$ contained in the stable locus.
The universal singular locus of $M$ coincides with the subvariety $M'$ of $M$ consisting of sheaves that are not locally free on their support. It turns out that the blow up $\Bl_{M'}M$ may be naturally seen as a compactification of $M_B=M\setminus M'$ by vector bundles (on support).
\end{abstract}
\maketitle
\tableofcontents

\section{Introduction}
%TODO
\subsection{Motivation}
C.~Simpson showed in~\cite{Simpson1} that
for an arbitrary projective variety $X$
%\mymargin{II.1 fixed}
and for an arbitrary
numerical  polynomial $P\in\bb Q[m]$ there is a coarse moduli space
$M\defeq M_P(X)$ of semi-stable sheaves on $X$ with Hilbert polynomial $P$,
which turns out to be a projective variety.

In general $M$ contains  a closed subvariety $M'$ of sheaves that are not locally free on their support.
%\mymargin{II.2 fixed}
Its complement $M_B$ is then  an open  subset whose points are sheaves that are locally free on their support. If $M$ is irreducible, then $M_B$ is dense and hence one could consider $M$ as a compactification of $M_B$. We call the sheaves
from the boundary $M\smallsetminus M_B$ \textit{singular}.
It is an interesting question whether and how one could replace the
boundary of singular sheaves by one which consists entirely of vector
bundles with varying and possibly reducible supports. This problem for one-dimensional sheaves on a projective plane was dealt with in~\cite{MyGermanDiss} and \cite{IenaTrmPaper1}. The case of torsion-free sheaves on a surface is considered in~\cite{TrautmannTikhMark}.

It is known (see~\cite{LePotier},~\cite{Freiermuth}) that the universal plane cubic curve may be identified with the fine Simpson moduli space of stable coherent sheaves on $\P_2$ with Hilbert polynomial $3m+1$. In~\cite{IenaTrmPaper1} it has been shown that the blowing-up of the universal plane cubic curve along its universal singular locus may be seen as a construction which substitutes the sheaves which are not vector bundles (on their $1$-dimensional support) by vector bundles (on support).

\subsection{Main results of the paper}
The main aim of this note is to show the following:
\begin{itemize}
\item  the universal plane curve of fixed degree $d\ge 3$ can be seen as a closed subvariety of codimension $\frac{d(d-3)}{2}$ in the Simpson moduli space of semistable sheaves on $\P_2$ with Hilbert polynomial
    \[
    dm+\frac{d(3-d)}{2}+1
    \]
     contained in the stable locus,
     this closed subvariety coincides with the locus of sheaves with non-zero global sections (Proposition~\ref{pr: subvariety in Simpson});
    \medskip
\item    the blowing up along the universal singular locus can be seen as a process which substitutes the singular sheaves, i.e., those which are not locally free on their support, by vector bundles (on support) (Theorem~\ref{tr: main}).
\end{itemize}
This generalizes the construction presented in~\cite{MyGermanDiss},~\cite{IenaTrmPaper1}.
Moreover, some important details omitted for the sake of brevity in~\cite{IenaTrmPaper1}  are presented (in bigger generality) here.

\subsection{Some notations and conventions.}
In this paper we use notations and constructions from~\cite{IenaTrmPaper1}, in particular $\k$ is an algebraically closed field of characteristic zero, we work in the category of separated schemes of finite type over $\k$ and
call them varieties, using only their closed points. We do not restrict ourselves to reduced or irreducible varieties.
Dealing with homomorphisms between direct sums of line bundles and identifying
them with matrices, we consider the matrices acting on elements from the
right, i.~e, the composition $X\xra{A}Y\xra{B}Z$ is given by the matrix
$A\cdot B$.

In~\cite{IenaTrmPaper1} surfaces $D(p)$ were defined for every point $p\in \P_2$.
$D(p)$ consists of two irreducible components $D_0(p)$ and $D_1(p)$,
$D_0(p)$ being the blow up of $\P_2$ at $p$ and $D_1(p)$
being another projective plane, such that these components intersect along the
line $L(p)$ which is the exceptional divisor of $D_0(p)$. Each surface
$D(p)$ can be defined as the subvariety in $\P_2\times\P_2$ with equations
$u_0x_1$, $u_0x_2$, $u_1x_2-u_2x_1$ where the $x_i$ respectively $u_i$ are
the homogeneous coordinates of the first respectively second $\P_2$, such that
the first projection contracts $D_1(p)$ to $p$ and describes $D_0(p)$
as the blow up. 
%The surface $D(p)$ can be also described as $\Proj(\k_p\oplus \cal I_{\{p\}})$, where $\cal I_{\{p\}}$ is the ideal sheaf of the one-point subscheme $\{p\}$ in $\P_2$ and $\k_p$ is its structure sheaf, i.~e., the skyscraper sheaf at $p$.

As in~\cite{IenaTrmPaper1}, $\O_{D(p)}(a,b)$ denotes the invertible sheaf
induced by $\O_{\P_2}(a)\sqten\O_{\P_2}(b)$.

\subsection{Structure of the paper}
In Section~\ref{section: quotient} we describe the universal curve as a quotient of a space of certain injective morphisms between rank $2$ vector bundles on $\P_2$. In Section~\ref{section: stability} we show that the universal curve is a subvariety of an appropriate Simpson moduli space. Proposition~\ref{pr: subvariety in Simpson} is proved here.
In Section~\ref{section: singular locus} we identify the universal singular locus with the  subvariety of singular sheaves in $M$.
In Section~\ref{section: blow up} we prove Theorem~\ref{tr: main}, i.~e., we show that the blowing up along the universal singular locus can be seen as a process which substitutes the singular sheaves by vector bundles (on support).

\subsection{Acknowledgements}
The author thanks Mario Maican for his valuable comments regarding Section~\ref{section: stability}. Many thanks as well to an unknown referee for helpful comments, improvements, and suggestions.

\section{Universal curve as a quotient}\label{section: quotient}
Let $V$ be a $3$-dimensional vector space over $\k$. Let $\P_2=\P V$ be the corresponding projective space. Let $S^d V$ be the $d$-th symmetric power of $V$. Then $\P_N=\P (S^dV^*)$
%\mymargin{II.3 fixed}
may be seen as the space of plane curves of degree $d$. Its dimension is  $N=\frac{(d+2)(d+1)}{2}-1$. Recall that a curve is identified with its equation up to multiplication by a non-zero constant. Assume that $d\ge 3$.

Consider the universal plane curve of degree $d$
\[
M=\{(C, p)\ |\ p\in C\}=
\{(\point{f}, \point{x})\in \P_N\times \P_2\ |\ f(x)=0\}.
\]
This is a smooth projective variety of dimension $N+1=\frac{(d+2)(d+1)}{2}$.

Let $X$ be the space of morphisms
$2\O_{\P_2}(-d+1)\xra{A}\O_{\P_2}(-d+2)\oplus \O_{\P_2}$,
 $A=\pmat{z_1&q_1\\z_2&q_2}$, with linear independent $z_1$ and $z_2$ and with non-zero determinant. Note that we consider matrices acting on the right.
Then one sees that $X$ is an open subvariety in the affine variety
$\Hom(2\O_{\P_2}(-d+1), \O_{\P_2}(-d+2)\oplus \O_{\P_2})$, which is isomorphic to $\k^{d^2+d+6}$.

We fix a basis $\{x_0, x_1, x_2\}$ of $H^0(\P_2, \O_{\P_2}(1))$ and for $A\in X$ we will write
\[
z_1=a_0x_0+a_1x_1+a_2x_2,\quad z_2=b_0x_0+b_1x_1+b_2x_2,
\]
\[
q_1=\sum_{\substack{i,j\ge 0,\\i+j\le d-1}} A_{ij}x_0^{d-1-i-j}x_1^ix_2^j, \quad
q_2=\sum_{\substack{i,j\ge 0,\\i+j\le d-1}} B_{ij}x_0^{d-1-i-j}x_1^ix_2^j.
\]
Since all morphisms
%\mymargin{II.4 fixed}
in $X$ are injective, $X$ may be seen as a parameter space of sheaves given by resolutions
\begin{equation}\label{eq:resolution one point}
0\ra 2\O_{\P_2}(-d+1)\xra{\smat{z_1&q_1\\z_2&q_2}}\O_{\P_2}(-d+2)\oplus \O_{\P_2}\ra \cal F\ra 0.
\end{equation}

\begin{rem}
One can easily see that the Hilbert polynomial of such sheaves is $dm+\frac{d(3-d)}{2}+1$.
\end{rem}

 There is a morphism $X\xra{\nu} M$, $\smat{z_1&q_1\\z_2&q_2}\mto (\point{z_1q_2-z_2q_1}, z_1\wedge z_2)$, where $z_1\wedge z_2$ denotes the common zero of $z_1$ and $z_2$.
 \begin{lemma}
 $\nu$ is surjective.
 \end{lemma}
 \begin{proof}
 Let $(f, p)\in M$. Choose two linear independent linear forms $z_1$ and $z_2$ such that $p=z_1\wedge z_2$. Since $f(p)=0$, one can write $f=z_1q_2-z_2q_1$ for some forms $q_1$ and $q_2$ of degree $d-1$. Then $\smat{z_1&q_1\\z_2&q_2}$ is a preimage of $(f, p)$.
 \end{proof}
 \begin{lemma}
 Two matrices $A_1, A_2\in X$ lie in the same fibre of $\nu$ if and only if there exist $g\in \GL_2(\k)$ and $h=\smat{\lambda &q\\0&\mu}\in \Aut(\O_{\P_2}(-d+2)\oplus \O_{\P_2})$ such that $gA_1h=A_2$.
 \end{lemma}
\begin{proof}
It is clear that $gA_1h=A_2$ implies that both $A_1$ and $A_2$ lie in the same fibre of $\nu$. Let us assume that $A_1$ and $A_2$ lie in the same fibre of $\nu$. Then in particular $pr_2\circ \nu(A_1)=pr_2\circ \nu (A_2)$, where $pr_2: M\subset \P_N\times \P_2\ra \P_2$ is the projection onto the second factor.
%\mymargin{II.5 fixed}
Multiplying  $A_1$ by an appropriate $g\in \GL_2(\k)$ we may assume that $A_1=\smat{z_1&q_1\\z_2&q_2}$ and  $A_2=\smat{z_1&q_1'\\z_2&q_2'}$. Since $\point{z_1q_2-z_2q_1}=\point{z_1q_2'-z_2q_1'}$ we obtain $z_1q_2-z_2q_1=\xi(z_1q_2'-z_2q_1')$ for some $\xi\in \k^*$ and thus multiplying the second column of  $A_1$ by $\xi$ we may assume that $\xi=1$. Then $z_1(q_2-q_2')-z_2(q_1-q_1')=0$ and  hence $\smat{q_1- q_1'\\q_2- q_2'}=q\cdot \smat{z_1\\z_2}$ for some form $q$ of degree $d-1$. In other words
\[
\smat{z_1&q_1\\z_2&q_2}=\smat{z_1&q_1'\\z_2&q_2'}\cdot\smat{1&q\\0&1}.
\]
This proves the lemma.
\end{proof}

Note that the group $G=\Aut(2\O_{\P_2}(-d+1))\times \Aut(\O_{\P_2}(-d+2)\oplus \O_{\P_2})$ naturally acts on $X$ by the rule $(g, h)\cdot A=gAh^{-1}$. The last two lemmas show that $M$ is an orbit space of this action.

\begin{lemma}
The stabilizer of an arbitrary element in $X$ coincides with the group
\[
St=\{(\smat{\lambda&0\\0&\lambda},(\smat{\lambda&0\\0&\lambda})\ |\ \lambda\in \k^*\}.
\]
\end{lemma}
\begin{proof}
Let $A=\smat{z_1&q_1\\z_2&q_2}$ and suppose $g$ and $h$ satisfy $gA=Ah$. Write $h=\smat{\lambda&q\\0&\mu}$. Then $g\smat{z_1\\z_2}=\lambda \smat{z_1\\z_2}$ and $(g-\lambda\cdot\id)\smat{z_1\\z_2}=0$. Since $z_1$ and $z_2$ are linear independent, one concludes that $g=\lambda\cdot \id$. Then $A=\lambda^{-1} A h$ and since $\det A\neq 0$ one concludes that $\lambda=\mu$ and $z_1q=z_2q=0$. Hence $h=\lambda\cdot \id$. This completes the proof.
\end{proof}
Denote the group $G/St$ by $\P G$. Then $\P G$ acts freely on $X$.
\begin{lemma}
There is a local section of $\nu$.
\end{lemma}
\begin{proof}
It is enough to prove the existence of a section $s$ over the points
$(\point{f}, p)=(\point{f}, \point{1, \xi, \eta})$.
Since $f(p)=0$, there is  a unique choice of polynomials $G$ and $H$, in two and three variables respectively, such that
\[
f(x_0, x_1, x_2)=(x_1-\xi x_0)G(x_0, x_1-\xi x_0)+(x_2-\eta x_0)H(x_0, x_1-\xi x_0, x_2-\eta x_0)
\]
hence one can define a local section of $\nu$  by the rule
\[
(\point{f}, \point{1, \xi, \eta})\mto \pmat{x_1-\xi x_0&-H(x_0, x_1-\xi x_0, x_2-\eta x_0)\\x_2-\eta x_0&G(x_0, x_1-\xi x_0)}.
\]
This proves the required statement.
%\mymargin{II.6 fixed}
\end{proof}

Using the existence of a local section and Zariski main theorem one shows that $X$ is a principal $\P G$-bundle over $M$. Hence $M$ is a geometrical quotient.

\begin{lemma}
Every morphism of two sheaves with resolution of the type~\refeq{eq:resolution one point} can be uniquely lifted to a morphism of resolutions.
\end{lemma}
\begin{proof}
Follows from
$\Ext^1(\O_{\P_2}(-d+2)\oplus\O_{\P_2}, 2\O_{\P_2}(-d+1))=
\Hom(\O_{\P_2}(-d+2)\oplus\O_{\P_2}, 2\O_{\P_2}(-d+1))=0$.
\end{proof}
This lemma implies the following.

\begin{pr}\label{pr: points of M and sheaves}
The points of $M$ are in one-to-one correspondence with the isomorphism classes of sheaves that possess resolutions of the type~\refeq{eq:resolution one point}.
\end{pr}

\begin{lemma}\label{lemma:resANDext}
Let $d\ge 3$. For a sheaf $\cal F$ on $\P_2$ the following conditions are equivalent.

1) $\cal F$ has a resolution of  the type~\refeq{eq:resolution one point}.

2) There is a point $(C, p)\in M$ such that $\cal F$ is a non-trivial extension
\[
0\ra \O_C\ra \cal F\ra \k_{p}\ra 0.
\]
\end{lemma}
\begin{proof}
To prove that 1) implies 2) it is enough to notice that there is the following commutative diagram with exact rows and columns with $p$ being the common zero of $z_1$ and $z_2$.
\[
\begin{xy}
(-45,25)*+{0}="-12";
(0,25)*+{0}="02";
(40,25)*+{}="12";
%%%%%%%%%%%%%%%%%%%%
(-65,15)*+{0}="-21";
(-45,15)*+{\O_{\P_2}(-d)}="-11";
(0,15)*+{\O_{\P_2}}="01";
(30,15)*+{\O_{C}}="11";
(40,15)*+{0}="21";
%%%%%%%%%%%%%%%%%%%
(-65,0)*+{0}="-20";
(-45,0)*+{2\O_{\P_2}(-d+1)}="-10";
(0,0)*+{\O_{\P_2}(-d+2)\oplus \O_{\P_2}}="00";
(30,0)*+{\cal F}="10";
(40,0)*+{0}="20";
%%%%%%%%%%%%%%%%%%%
(-65,-15)*+{0}="-2-1";
(-45,-15)*+{\cal A}="-1-1";
(0,-15)*+{\O_{\P_2}(-d+2)}="0-1";
(30,-15)*+{\k_p}="1-1";
(40,-15)*+{0.}="2-1";
%%%%%%%%%%%%%%%%%%%%%
(-45,-25)*+{0}="-1-2";
(0,-25)*+{0}="0-2";
(30,-25)*+{}="1-2";
%%%%%%%%%%%%%%%%%%%%%
{\ar@{->}"-21";"-11"};
{\ar@{->}^-{z_1q_2-z_2q_1}"-11";"01"};
{\ar@{->}"01";"11"};
{\ar@{->}"11";"21"};
%%%%%%%%%%%%%%%%%%%%%%%
{\ar@{->}"-20";"-10"};
{\ar@{->}^-{\smat{z_1&q_1\\z_2&q_2}}"-10";"00"};
{\ar@{->}"00";"10"};
{\ar@{->}"10";"20"};
%%%%%%%%%%%%%%%%%%%%%%%
{\ar@{->}"-2-1";"-1-1"};
{\ar@{->}"-1-1";"0-1"};
{\ar@{->}"0-1";"1-1"};
{\ar@{->}"1-1";"2-1"};
%%%%%%%%%%%%%%%%%%%%%%%%
{\ar@{->}"-12";"-11"};
{\ar@{->}_{\smat{-z_2&z_1}}"-11";"-10"};
{\ar@{->}"-10";"-1-1"};
{\ar@{->}"-1-1";"-1-2"};
%%%%%%%%%%%%%%%%%%%%%%%%%%%
{\ar@{->}"02";"01"};
{\ar@{->}^{\smat{0&1}}"01";"00"};
{\ar@{->}^{\smat{1\\0}}"00";"0-1"};
{\ar@{->}"0-1";"0-2"};
%%%%%%%%%%%%%%%%%%%%%%%%%
{\ar@{->}_{\smat{z_1\\z_2}}"-10";"0-1"};
%%%%%%%%%%%%%%%%%%%%%%%%
\end{xy}
\]
This induces an extension
\(
0\ra \O_C\ra \cal F\ra \k_{p}\ra 0.
\)
Every splitting $\k_p\ra \cal F$ lifts to a morphism between the exact sequences
\[
2 \O_{P_2}(-d+1)\xra{\smat{z_1\\z_2}} \O_{P_2}(-d+2)\ra k_p\ra 0
\]
and~\refeq{eq:resolution one point}. Using that $z_1$ and $z_2$ are linear independent and $z_1q_2-z_2q_2\neq 0$, one concludes that the splitting is zero, which is a contradiction. Therefore, $\cal F$ is a non-trivial extension.

To prove another implication it is enough to modify the second part of the proof of~\cite[Lemma~5.3]{Freiermuth-Diplom} (so called Horseshoe Lemma argument).
\end{proof}
From Lemma~\ref{lemma:resANDext} one obtains the following reformulation of Proposition~\ref{pr: points of M and sheaves}.
\begin{cor}\label{cor: points of M and extensions}
The points of $M$ are in one-to-one correspondence with the isomorphism classes of
the non-trivial extensions
\[
0\ra \O_C\ra \cal F \ra \k_p\ra 0,\quad (C, p)\in M.
\]
\end{cor}

\section{Stability. Universal curve as a subvariety of Simpson moduli space}\label{section: stability}

The following is a partial case of~\cite[Proposition~1]{MaicanHilb}.
\begin{lemma}\label{lemma:ideal_of_point_on_curve}
 Let $C$ be a plane projective  curve of degree $d$ and let $p$ be a point at $C$. Then the ideal sheaf of a point $p$ at $C$, i.~e., the sheaf $\cal I$ given by the exact sequence
\[
0\ra \cal I\ra \O_C\ra \k_p\ra 0
\]
is stable.
\end{lemma}
\begin{proof}
For a coherent sheaf $\cal F$ on $\P_2$ we denote its Hilbert polynomial and the corresponding reduced Hilbert polynomial by $P_{\cal F}(m)$ and  $p_{\cal F}(m)$ respectively.

Since $\cal O_C$ does not have zero dimensional torsion, the same is true for its subsheaf $\cal I$, i.~e., $\cal I$ is pure-dimensional.

Let $\cal E$ be a proper subsheaf of $\cal I$. Since $\cal E$ is one-dimensional, its Hilbert polynomial is $am+b$. If the multiplicity $a$ equals $d$, then $\cal I/\cal E$ is zero dimensional, hence $P_{\cal I}(m)-P_{\cal E}(m)=h^0(\cal I/\cal E)>0$ and $p_{\cal E}(m)<p_{\cal I}(m)$.

 Assume $a<d$. Then using  \cite[Lemma~6.7]{MaicanTwoSemiSt}, we obtain a curve $S\subset C$ of degree $s<d$ such that its ideal sheaf $\cal I_S\subset \O_C$ contains $\cal E$ and $\cal Q\defeq\cal I_S/\cal E$ is a zero-dimensional sheaf.
Hence the Hilbert polynomial of $\cal E$ is
\begin{align*}
P_{\cal E}(m)=&P_{\cal I_S}(m)-h^0(\cal Q)=P_{\O_C}(m)-P_{\O_S}(m)
-h^0(\cal Q)=\\
&dm+\frac{d(3-d)}{2}-(sm+\frac{s(3-s)}{2})-h^0(\cal Q)
\end{align*}
Therefore,
\[
p_{\cal E}(m)=m+\frac{3}{2}-\frac{d+s}{2}-\frac{h^0(\cal Q)}{d-s}.
\]
Since
\[
p_{\cal I}(m)=m+\frac{(3-d)}{2}-\frac{1}{d},
\]
one sees that $p_{\cal E}(m)<p_{\cal I}(m)$ if and only if $\frac{1}{d}<\frac{s}{2}+\frac{h^0(\cal Q)}{d-s}$ or equivalently $1<\frac{sd}{2}+d\cdot\frac{h^0(\cal Q)}{d-s}$, which is clearly true since $d\ge 3$.

We proved $p_{\cal E}(m)<p_{\cal I}(m)$ for every proper subsheaf $\cal E$ of $\cal I$. Therefore, $\cal I$ is stable.
\end{proof}

\begin{pr}\label{pr: subvariety in Simpson}
1) The sheaves with resolution~\refeq{eq:resolution one point} are stable.

2) The corresponding map
\[
M\ra M_{dm+\frac{d(3-d)}{2}+1}(\P_2), \quad (\point{f}, p)\mto [\cal F],
\]
%\mymargin{II.8  fixed, probably needs to be better formulated }
is a closed embedding of codimension $\frac{d(d-3)}{2}$.

3)\footnote{This observation is due to an unknown referee.} The image of the embedding from 2) coincides with the locus
\[
\{[\cal E]\in M_{dm+\frac{d(3-d)}{2}+1}(\P_2) \mid h^0(\cal E)\neq 0\}
\]
of sheaves with global sections.
\end{pr}
\begin{proof}
1) The isomorphism class of every sheaf $\cal F$ with resolution~\refeq{eq:resolution one point} is represented by a plane projective  curve $C$ of degree $d$ and a point $p$ at $C$. By Lemma~\ref{lemma:resANDext}  $\cal F$ is a non-trivial extension
\[
0\ra \O_C\ra \cal F \ra \k_p\ra 0
\] and can be
obtained as $\cExt^1(\cal I, O_{\P_2})(-d)$, where $\cal I$ is the ideal sheaf of a point $p$ at $C$, i.~e., $\cal I$ is given by the exact sequence
\[
0\ra \cal I\ra \O_C\ra \k_p\ra 0.
\]
Using the notation  $\cal I^D\defeq \cExt^1(\cal I, \omega_{\P_2})$ from~\cite{MaicanDuality} we get $\cal F=\cal I^D(-d+3)$. By Lemma~\ref{lemma:ideal_of_point_on_curve} $\cal I$ is stable. Therefore, by the result from~\cite{MaicanDuality} its dual  $\cal I^D$ is stable as well. Note that by~\cite[Lemma~9.2]{MaicanTwoSemiSt} it can not be properly semi-stable.

This proves the first part of the statement.

2) There is  a family of sheaves, flat over $X$, with Hilbert polynomial $dm+\frac{d(3-d)}{2}+1$ given by the resolution
\[
0\ra 2\O_{X\times\P_2}(-d+1)\xra{\Psi}\O_{X\times\P_2}(-d+2)\oplus \O_{X\times\P_2}\ra \fcal F\ra 0,
\]
where $\Psi|_{\{A\}\times\P_2}=A$. Since $X\xra{\nu}M$ is a $\P G$-bundle over $M$, we get locally over $M$ a flat family of sheaves, which induces the required morphism.

Clearly the morphism is injective with a closed image. It remains to show that it is a closed embedding. In other words we need to consider its image equipped with the induced structure and show that the inverse map is a morphism. This can be done using the method from~\cite[6.5]{Maican6m+1}.
Namely,
given a point in $M_{dm+\frac{d(3-d)}{2}+1}(\P_2)$ represented by a sheaf $\cal F$ with resolution \refeq{eq:resolution one point}, it is enough to construct a point
of the universal curve of degree $d$ from the Beilinson spectral sequence converging to $\cal F$ (cf.~\cite[3.1.4.~Theorem~II, page~245]{Okonek-Schneider-Spindler} and also~\cite{BeilinsonEng}) by means of algebraic operations.

Since $\cal F$ is a one-dimensional sheaf, the only non-trivial part of the first sheet of the Beilinson spectral sequence
\[
E_1^{p,q}(\cal F)=H^q(\P_2, \cal F\ten \Omega^{-p}(-p))\ten O(p)
\] is a $2\times 3$  rectangular

\[
\begin{xy}
(0,0)*+{E_1^{0,0}.}="00";
(-20,0)*+{E_1^{-1,0}}="-10";
(-40,0)*+{E_1^{-2,0}}="-20";
(0,15)*+{E_1^{0,1}}="01";
(-20,15)*+{E_1^{-1,1}}="-11";
(-40,15)*+{E_1^{-2,1}}="-21";
{\ar@{->}"-20";"-10"};
{\ar@{->}"-10";"00"};
{\ar@{->}"-21";"-11"};
{\ar@{->}"-11";"01"};
\end{xy}
\]
Analyzing this spectral sequence as in~\cite[2.2]{DrezetMaican4m}, basically repeating the proof of~\cite[Proposition~2]{MaicanHilb}, and taking into account the stability of $\cal F$ one can conclude that  $\cal F$ is a non-trivial extension
\[
0\ra \O_C\ra \cal F \ra \k_p\ra 0,
\]
where $(C, p)\in M$ and the sheaves $\O_C$, $\k_p$ can be computed in terms of cokernels of the maps involved in the Beilinson spectral sequence.
%\mymargin{Are more details needed?}

3) Clearly, every sheaf with resolution~\refeq{eq:resolution one point} has a non-trivial section.

Let now $[\cal E]$ be a point in $M_{dm+\frac{d(3-d)}{2}+1}(\P_2)$ with a non-trivial section $\O_{\P_2}\xra{s}\cal E$. The kernel of $s$ is an ideal sheaf of a subscheme $Y$ in $\P_2$, hence we  obtain an injection $\O_Y\ra \cal E$ and conclude that $\O_Y$ is pure-dimensional as a  subsheaf of a pure-dimensional sheaf $\cal E$. Therefore, $Y$ is a curve. Let $a=\deg Y$, then the Hilbert polynomial of $\O_Y$ is $am+\frac{a(3-a)}{2}$. By the semi-stability of $\cal E$ we get
\(
m+\frac{3-a}{2}\le m+\frac{3-d}{2}+\frac{1}{d}
\)
and therefore
\(
a\ge d-\frac{2}{d}.
\)
Since $d\ge 3$, this means $a> d-1$ and thus $a=d$. So $Y$ is a curve of degree $d$ and the Hilbert polynomial of  the quotient sheaf $\cal E/\O_Y$ equals $1$. Therefore, the quotient is   isomorphic to a skyscraper sheaf $\k_p$ for some point $p\in Y$. We obtained an extension
\[
0\ra \O_Y\ra \cal E\ra \k_p\ra 0,
\]
which is non-trivial by the semi-stability of $\cal E$.
\end{proof}

\section{Universal singular locus as the subvariety of singular sheaves}\label{section: singular locus}
Let $X'$  be the subvariety of matrices in $X$ defining singular sheaves, i.e., sheaves that are not locally free on their support.

A matrix $A\in X$ as in ~\refeq{eq:resolution one point} defines a singular sheaf if and only if it vanishes at some point $q$ of $\P_2$. Since the linear forms $z_1$ and $z_2$ are linear independent, this point could only be the common zero point of $z_1$ and $z_2$. If $z_1=a_0x_0+a_1x_1+a_2x_2$ and $z_2=b_0x_0+b_1x_1+b_2x_2$, then $q=\point{d_0, d_1, d_2}$, where $d_i$ are the minors of the matrix $\smat{a_0&a_1&a_2\\b_0&b_1&b_2}$. Hence $X'$ is given as a closed subvariety in $X$ by the equations
\begin{equation}\label{eq:equations X'}
f_1=q_1(d_0, d_1, d_2)=0,\quad f_2=q_2(d_0, d_1, d_2)=0.
\end{equation}
After computing the partial derivatives of $f_1$ and $f_2$ and taking into account that the minors $d_0$, $d_1$, and $d_2$ do not vanish simultaneously since $z_1$ and $z_2$ are always linear independent, we conclude that $X'$ is a smooth subvariety of codimension $2$ in $X$.

\begin{lemma}
A point $(C, p)$ from $M$ corresponds to a singular sheaf if and only if $p$ is a singular point of $C$, i.e., the subvariety $M'$ of singular sheaves coincides with the universal singular locus $\{(C, p)\ |\ p\in\Sing(C)\}$.
\end{lemma}
\begin{proof}
Let $(C, p)$ be a point in $M$. Then there is a matrix $A=\smat{z_1&q_1\\z_2&q_2}\in X$ such that $C$ is the zero set of $f=\det A$ and $p$ is the common  zero set of $z_1$ and $z_2$.

Suppose $(C, p)$ corresponds to a singular sheaf.  Then $q_1(p)=q_2(p)=0$ and one checks that $\dd{f}{x_i}(p)=0$ for all $i=0,1,2$.

If $p$ is a singular point of $C$, then all partial derivatives $\dd{f}{x_i}(p)$ vanish. Since $\dd{f}{x_i}(p)=(\dd{z_1}{x_i}q_2-\dd{z_2}{x_i}q_1)(p)$ and since $z_1$ and $z_2$ are linear independent, one concludes that $q_1(p)=q_2(p)=0$, hence $(C, p)$ defines a singular sheaf.
\end{proof}

Since $X$ is a principal bundle over $M$ and since $X'$ is smooth, one concludes that $M'$ is smooth as well. One can also show this directly. The codimension of $M'$ in $M$ is $2$.

Let $M_B=M\setminus M'$, its points are isomorphism classes of vector bundles (on support). Then one could consider $M$ as a compactification of $M_B$ by coherent sheaves.

\section{Blow up $\Bl_{M'}(M)$ as a compactification of $M_B$ by vector bundles}
\label{section: blow up}

For a fixed point $(C, p)\in M'$ representing an isomorphism class $[\cal F]$ of a singular sheaf and for a fixed tangent vector $v\in T_{[\cal F]}M\setminus T_{[\cal F]}M'$, i.~e., $v$ is normal to $M'$, we are going to construct a $1$-dimensional sheaf  on the surface $D(p)$, locally free on its support. We call such sheaves $R$-bundles.
%%new
They are flat degenerations of the non-singular sheaves represented by the points of $M\setminus M'$.
%%end_new
We are going to show that $\P (T_{[\cal F]} M/T_{\cal F}M')$ is naturally the space of equivalence classes of $R$-bundles.
We shall use the parameter space $X$.

Let $A\in X'$ and $B\in T_A X\setminus T_A X'$ represent a singular sheaf $[\cal F]=(C, p)\in M'$ and a tangent vector at $[\cal F]$ normal to $M'$ respectively.

Since $X$ is an open subset of the affine variety $\bb A=\Hom(2\O_{\P_2}(-d+1), \O_{\P_2}(-d+2)\oplus \O_{\P_2})$,
 we can identify $T_A X$ with $\bb A$. Let $T\subset \k$ be the preimage of $X$ under the morphism  $\k\ra \bb A$, $t\mto A+tB$.

Then the cokernel  $\fcal F$ of the injective morphism
\[
2\O_{T\times\P_2}(-d+1)\xra{A+tB}\O_{T\times\P_2}(-d+2)\oplus \O_{T\times\P_2},
\]
is a flat family of sheaves in $M$ (considered as a subvariety in the corresponding moduli space by Proposition~\ref{pr: subvariety in Simpson}) such that
the restriction of $\fcal F$ to the fibre $\{0\}\times \P_2$ is isomorphic to $\cal F$. Since $B$ does not belong to $T_A X'$,
the restrictions $\fcal F_t$ to the fibres $\{t\}\times \P_2$ are non-singular sheaves for $t\neq 0$ in a neighbourhood of zero. Shrinking $T$ if necessary, we can assume that $\fcal F_t$ are non-singular sheaves for all $t\in T$, $t\neq 0$.

Let $Z\xra{\sigma} T\times \P_2$ be the blowing up $Z=\Bl_{0\times p} (T\times \P_2)$. Let $D_1=D_1(p)$ be its exceptional divisor and let $s$ be the canonical section of $\O_Z(D_1)$.
Let $x_i$ denote the homogeneous coordinates
of $\P_2$, such that the point $0\times p$ has the equations $tx_0, x_1, x_2$. Then
$Z$ is embedded in $T\times\P_2\times\P_2$ with equations
\[
tx_0u_1-x_1u_0,\quad tx_0u_2-x_2u_0,\quad x_1u_2-x_2u_1,
\]
where the $u_i$ are the coordinates of the second  $\P_2$. Note that $s$ is locally given by $\frac{tx_0}{u_0}$, $\frac{x_1}{u_1}$, or $\frac{x_2}{u_2}$. This implies $\O_Z(D_1)\iso \O_Z(1, -1)$. Moreover,
\begin{equation}\label{eq:factoriz}
tx_0=su_0,\quad x_1=su_1,\quad x_2=su_2
\end{equation}
if one considers  $x_i$ as global sections of $\O_Z(1, 0)$ and $u_i$ as global sections of $\O_Z(0, 1)$.

Note that the morphism $Z\xra{\sigma} T\times \P_2\xra{pr_1} T$ is flat.
Indeed,
since both $Z$ and $T$ are regular, $\dim Z=3$, $\dim T=1$, and $\dim
Z_t=2=\dim Z-\dim T$ for all $t\in T$, this follows
from~\cite[6.1.5.]{EGA-IV-II}. Notice that the fibres over $t\neq 0$ are isomorphic to $\P_2$ and $Z_0$ is isomorphic to $D(p)$.

By the construction  the pullback $\sigma^*(A+tB)$ vanishes at $D_1$ and hence can be factored as
\begin{equation}\label{eq:factorizRES}
%\sigma^*(A+tB)=
2O_Z(-d+1, 0)\xra{\smat{s&0\\0&s}} 2O_Z(-d+2, -1)  \xra{\phi(A, B)}O_Z(-d+2, 0)\oplus \O_Z.
\end{equation}
Let $\fcal E$ be the cokernel of $\phi(A, B)$.

Since $Z$ is isomorphic to $T\times \P_2$  outside of the fibre $Z_0$, the sheaf $\fcal E$ can be seen as a family of non-singular sheaves in $M$ parameterized by $T\setminus \{0\}$. Hence the sheaf $\cal E=\cal E(A, B)=\fcal E_0$ on $D(p)=Z_0$ is a degeneration of non-singular sheaves. Let $\Phi(A, B)$ denote the restriction of $\phi(A,B)$ to $Z_0$, we obtain a locally free resolution of $\cal E(A, B)$
\[
2\O_{D(p)}(-d+2, -1)\xra{\Phi(A, B)}  \O_{D(p)}(-d+2, 0)\oplus \O_{D(p)}\ra \cal E(A, B)\ra 0.
\]
Assume without loss of generality that
\begin{equation}\label{eq:A special}
\quad p=\point{1, 0, 0},\quad
A=\pmat{x_1&q_1\\x_2&q_2}\in X'.
\end{equation}
We can write $A$ as
\[
\pmat{
x_1&A_{10}x_0^{d-2}x_1+A_{01}x_0^{d-2}x_2+x_1^2 P_1(x_0, x_1)+ x_1x_2Q_1(x_0, x_1, x_2)+x_2^2R_1(x_0,x_2) \\
x_2&B_{10}x_0^{d-2}x_1+B_{01}x_0^{d-2}x_2+x_1^2 P_2(x_0, x_1)+ x_1x_2Q_2(x_0, x_1, x_2)+x_2^2R_2(x_0,x_2)
},
\]
where $A_{10}, A_{01}, B_{10}, B_{01}\in \k$, $P_i\in \k[x_0, x_1]$, $Q_i\in \k[x_0, x_1, x_2]$, $R_i\in \k[x_0, x_2]$, $i=1,2$.
 Straightforward calculations using~\refeq{eq:equations X'} show  that the tangent equation at $A$ in this case
are
\begin{equation}\label{eq:tangent equations}
\begin{cases}
\xi_{00}= A_{10} \xi_0+ A_{01}\eta_0\\
\eta_{00}= B_{10} \xi_0 + B_{01}\eta_0
\end{cases},
\end{equation}
where
\[
B= \pmat{
\xi_0x_0+\xi_1x_1+\xi_2x_2   &\xi_{00}x_0^{d-1}+\dots+\xi_{0d-1}x_2^{d-1} \\
\eta_0x_0+\eta_1x_1+\eta_2x_2& \eta_{00}x_0^{d-1}+\dots+\eta_{0d-1}x_2^{d-1}
}
\]
is a tangent vector at $A$.

Then using~\refeq{eq:factoriz} one computes
\[
\Phi(A, B)=\pmat{
u_1&u_1A_{10}x_0^{d-2}+u_2A_{01}x_0^{d-2}+u_1x_1P_1+u_1x_2Q_1+u_2x_2R_1
\\
u_2&u_1B_{10}x_0^{d-2}+u_2B_{01}x_0^{d-2}+u_1x_1P_2+u_1x_2Q_2+u_2x_2R_2
}+
\pmat{\xi_0&\xi_{00}x_0^{d-2}\\\eta_0&\eta_{00}x_0^{d-2}}u_0.
\]
Note that $u_1x_2=u_2x_1$, hence there is no asymmetry in the formula.
The cokernel of such a matrix is not a locally free sheaf on its support (defined by the determinant of the matrix) if and only if all entries of the matrix vanish at some point. Since by the construction this can only happen on $D_1(p)$, one sees that this condition  is equivalent to
\[
\xi_{00}= A_{10} \xi_0+ A_{01}\eta_0, \quad \eta_{00}= B_{10} \xi_0 + B_{01}\eta_0.
\]
The latter are just the tangent equations of $X'$ at $A$.

So, for every $A\in X'$ and for every $B\in T_AX\setminus T_AX'$ we obtain a sheaf $\cal E=\cal E(A, B)$ on $D(p)$ locally free on its support. Here $p$ is the common zero of $z_1$ and $z_2$. We will call such sheaves \emph{$R$-bundles}.
By the following lemma $R$-bundles are flat limits of non-singular sheaves parameterized by the points of $M\setminus M'$.

\begin{lemma}
$\Phi(A, B)$ is injective for $B\in T_AX\setminus T_AX'$, hence $\fcal E$ is flat over $T$ and  $\cal E(A, B)$ is
given by the locally free resolution
\begin{equation}\label{eq:Phi(A, B)}
0\ra 2\O_{D(p)}(-d+2, -1)\xra{\Phi(A, B)}  \O_{D(p)}(-d+2, 0)\oplus \O_{D(p)}\ra \cal E(A, B)\ra 0.
\end{equation}
\end{lemma}
\begin{proof} Since the middle term of the exact sequence
\[
0\ra 2O_Z(-d+2, -1)  \xra{\phi(A, B)}O_Z(-d+2, 0)\oplus \O_Z\ra \fcal E \ra 0.
\]
is flat over $T$, one concludes that $\fcal E$ is flat if and only if for every point $t\in T$ the restriction of the exact sequence to the fibres $Z_t$ remains exact. Since this  is clearly the case over $t\neq 0$, the flatness of $\fcal E$ follows from the injectivity of $\Phi(A, B)$.

Notice that it is enough to show that the restrictions of $\Phi(A, B)$ to its components $D_0(p)$ and $D_1(p)$ are injective. In this case the kernel of $\Phi(A, B)$  can only be  supported  on  $L=D_0(p)\cap D_1(p)$. Hence it should be zero because locally free sheaves have no torsion.

 Since $A$ is injective, we immediately conclude that the kernel of $\Phi(A, B)|_{D_0(p)}$ can only be supported on $L$. On the other hand, locally free sheaves on $D_0(p)$ are torsion free, hence $\Phi(A, B)|_{D_0(p)}$ is injective.

Under assumptions of \refeq{eq:A special} the restriction of $\Phi(A, B)$ to $D_1(p)\iso \P_2$ is given by the matrix
\[
\pmat{
u_1+\xi_0 u_0 &u_1A_{10}+u_2A_{01}+\xi_{00}u_0
\\
u_2+\eta_0 u_0&u_1B_{10}+u_2B_{01}+\eta_{00}u_0
}.
\]
Vanishing of its determinant implies $B\in T_A X'$, which is a contradiction. Therefore, the restriction of  $\Phi(A, B)$ to $D_1(p)$ is injective.
\end{proof}
As the following lemma shows,
the $R$-bundles are not only new flat limits of non-singular sheaves but they preserve the information about the singular sheaves as well.
\begin{lemma}
$\sigma_*\fcal E \iso \fcal F$.
\end{lemma}
\begin{proof}
Notice that by the construction of $\fcal E$ there is an exact sequence
\[
0\ra \cal C\ra \sigma^*\fcal F\ra\fcal E\ra 0,
\]
where $\cal C=2\O_{D_1}(-d+2, -1)\iso 2\O_{D_1}(0, -1)\iso 2\O_{D_1}(-L)$ is the cokernel of the morphism
$\smat{s&0\\0&s}$ from~\refeq{eq:factorizRES}. Since $R^i\sigma_* \cal C=0$ for all $i\ge 0$, one concludes $\sigma_*\sigma^*\fcal F\iso\sigma_*\fcal E$. On the other hand, from the properties of blow-ups it follows that $\sigma_*\O_Z\iso \O_{T\times \P_2}$ and $R^i\sigma_*\O_Z=0$ for $i>0$, which implies that $\sigma_*\sigma^* \cal B\iso \cal B$ and $R^i\sigma_*\sigma^* \cal B=0$, $i>0$, for every locally free sheaf $\cal B$ on $T\times \P_2$.
Applying the functor $\sigma_*\sigma^*$ to the locally free resolution of $\fcal F$,
%(the one given by $A+tB$)
one gets $\sigma_*\sigma^*\fcal F\iso \fcal F$ and hence the required statement.
\end{proof}

\begin{lemma}\label{lemma:lift R-bundles}
Every morphism of two sheaves with resolution of the type~\refeq{eq:Phi(A, B)} can be uniquely lifted to a morphism of resolutions.
\end{lemma}
\begin{proof}
Follows from
$\Ext^1(\O_{D(p)}(-d+2, 0)\oplus\O_{D(p)}, 2\O_{D(p)}(-d+2, -1))=
\Hom(\O_{D(p)}(-d+2, 0)\oplus\O_{D(p)}, 2\O_{D(p)}(-d+2, -1))=0$.

We are going to prove that the groups
$H^0(D(p), \O_{D(p)}(-d+2, -1))$, $H^1(D(p), \O_{D(p)}(-d+2, -1))$, $H^0(D(p), \O_{D(p)}(0, -1))$, and $H^1(D(p), \O_{D(p)}(0, -1))$ are zero.

Let us compute the cohomology groups of the sheaf $\O_{D(p)}(-d+2, -1)$.
Consider the gluing exact sequence
\begin{equation*}
0\ra \O_{D(p)}(-d+2, -1)\ra \O_{D_0(p)}(-d+2, -1)\oplus \O_{D_1(p)}(-1)\ra \O_L(-1)\ra 0.
\end{equation*}
Since all the cohomology groups  of $\O_{D_1}(-L)$ and $\O_L(-1)$ are zero, using the long exact cohomology sequence we conclude that $H^i(D(p), \O_{D(p)}(-d+2, -1))\iso H^i(D_0(p), \O_{D_0(p)}(-d+2, -1))$.
From the exact sequence
\begin{equation*}
0\ra \O_{\P_2\times \P_1}(-d+1,-2)\ra \O_{\P_2\times \P_1}(-d+2, -1) \ra \O_{D_0(p)}(-d+2, -1)\ra 0
\end{equation*}
and the corresponding long exact cohomology sequence using that
\[
H^0(\P_2\times \P_1,\O_{\P_2\times \P_1}(-d+2, -1))=0, \quad H^1(\P_2\times \P_1,\O_{\P_2\times \P_1}(-d+1, -2))=0
\]
we conclude that $H^0(D_0(p), \O_{D_0(p)}(-d+2, -1))=0$.
Using that
\[
H^1(\P_2\times \P_1,\O_{\P_2\times \P_1}(-d+2, -1))=0, \quad H^2(\P_2\times \P_1,\O_{\P_2\times \P_1}(-d+1, -2))=0,
\]
we conclude that $H^1(D_0(p), \O_{D_0(p)}(-d+2, -1))=0$.

Analogously one computes that the cohomology groups of $\O_{D(p)}(0, -1)$ are zero as well.
\end{proof}
\begin{rem}
Note that the uniqueness of the lifting implies that the lifting of an isomorphism of $R$-bundles is an isomorphism in each degree.
\end{rem}

\begin{df}\label{df:equivalence R bundles}
Let $\cal E_1=\cal E(A, B_1)$ and $\cal E_2=\cal E(A, B_2)$ be two $R$-bundles on $D(p)$.
 We call them equivalent if there exists an automorphism $\phi$ of
$D(p)$ that acts identically on $D_0(p)$  and such that
$\phi^*(\cal E_1)\iso \cal E_2$.
\end{df}

\begin{pr}
Two $R$-bundles $\cal E_1=\cal E(A, B_1)$ and $\cal E_2=\cal E(A, B_2)$ are equivalent if and only if $B_1$ and $B_2$ represent the same point in $\P N_A$, where $N=T_AX/T_AX'$.
\end{pr}
\begin{proof}
{``$\Rightarrow$''.}
Let $\cal E_1=\cal E(A, B_1)$ and $\cal E_2=\cal E(A, B_2)$ be two equivalent
$R$-bundles, then the sheaves $\cal E_1$ and $\cal E_2$ possess locally free
resolutions of type~\refeq{eq:Phi(A, B)},
they are cokernels of $\Phi_1=\Phi(A, B_1)$ and $\Phi_2=\Phi(A, B_2)$ respectively.

Equivalence of $\cal E_1$  and $\cal E_2$ means that there exists an isomorphism $\phi:D(p)\ra D(p)$ identical on
$D_0(p)$ such that there is an isomorphism $\cal E_2\xra{\xi} \phi^*(\cal E_1)$.
By Lemma~\ref{lemma:lift R-bundles} $\xi$ can be uniquely lifted to a morphism of resolutions
\begin{equation}\label{eq:morph}
\begin{xy}
(5,0)*+{\O_{D(p)}(-d+2, 0)\oplus \O_{D(p)}}="1";
(38,0)*+{\cal E_2}="2";
(5,-15)*+{\O_{D(p)}(-d+2, 0)\oplus \O_{D(p)}}="3";
(38,-15)*+{\phi^*(\cal E_1)}="4";
{\ar@{->}"1";"2"};
{\ar@{->}"3";"4"};
{\ar@{->}^{
\smat{\bar a& \bar b \\0&\bar d}
}"1";"3"};
{\ar@{->}^{\xi}"2";"4"};
(-45,0)*+{2\O_{D(p)}(-d+2, -1)}="5";
(-45,-15)*+{2\O_{D(p)}(-d+2, -1)}="6";
{\ar@{->}^{\smat{a&b\\c&d}}"5";"6"};
{\ar@{->}^-{\phi^*(\Phi_1)}"6";"3"};
{\ar@{->}^-{\Phi_2}"5";"1"};
(-70, 0)*+{0}="7";
(-70,-15)*+{0}="8";
{\ar@{->}"8";"6"};
{\ar@{->}"7";"5"};
(50,0)*+{0}="9";
(50,-15)*+{0.}="10";
{\ar@{->}"2";"9"};
{\ar@{->}"4";"10"};
\end{xy}
\end{equation}
Note that from the uniqueness of the lifting it follows that  both matrices $\smat{a&b\\c&d}$ and $\smat{\bar a& \bar b \\0&\bar d}$ are invertible. Let $\bar b=\sum \bar b_{ij} x_0^{d-2-i-j}x_1^ix_2^j$.

We are going to show now that for some $\mu\in \k^*$ the matrix
 $B_2-\mu B_1$ satisfies the tangent equations~\refeq{eq:tangent equations},
 i.~e., $B_2-\mu B_1\in T_A(X_8)$. So $B_1$
and $B_2$ represent the same element in $\P N_A$.
Let us present here a detailed proof.

One can assume without loss of generality that $A$ is as in~\refeq{eq:A special}.
Let
\[
\smat{\alpha&\beta&\gamma\\
0&1&0\\
0&0&1
}:\P_2\ra \P_2,\quad \point{u_0, u_1, u_2}\mto \langle(u_0, u_1, u_2)
\smat{\alpha&\beta&\gamma\\
0&1&0\\
0&0&1
}\rangle
\]
be the restriction of $\phi$ to $D_0(p)$.
Let
\[
B_1= \pmat{
\xi_0x_0&\xi_{00}x_0^{d-1} \\
\eta_0x_0& \eta_{00}x_0^{d-1}
}\mod (x_1, x_2), \quad  B_2=
\pmat{
\mu_0x_0&
\mu_{00}x_0^{d-1}\\
\nu_0x_0&
\nu_{00}x_0^{d-1}
}\mod (x_1, x_2).
\]
Then
\[
\Phi_1=\pmat{
u_1&u_1A_{10}x_0^{d-2}+u_2A_{01}x_0^{d-2}+u_1x_1P_1+u_1x_2Q_1+u_2x_2R_1
\\
u_2&u_1B_{10}x_0^{d-2}+u_2B_{01}x_0^{d-2}+u_1x_1P_2+u_1x_2Q_2+u_2x_2R_2
}+
\pmat{\xi_0&\xi_{00}x_0^{d-2}\\\eta_0&\eta_{00}x_0^{d-2}}u_0,
\]
\[
\Phi_2=\pmat{
u_1&u_1A_{10}x_0^{d-2}+u_2A_{01}x_0^{d-2}+u_1x_1P_1+u_1x_2Q_1+u_2x_2R_1
\\
u_2&u_1B_{10}x_0^{d-2}+u_2B_{01}x_0^{d-2}+u_1x_1P_2+u_1x_2Q_2+u_2x_2R_2
}+
\pmat{\mu_0&\mu_{00}x_0^{d-2}\\
\nu_0&\nu_{00}x_0^{d-2}}u_0,
\]
and using $u_0x_1=u_0x_2=0$ we conclude that $\phi^*(\Phi_1)$ equals
\begin{equation}\label{eq:pull back matrix}
\begin{split}
\pmat{
u_1+\beta u_0 &(u_1+\beta u_0)A_{10}x_0^{d-2}+(u_2+\gamma u_0)A_{01}x_0^{d-2}
\\
u_2+\gamma u_0&(u_1+\beta u_0)B_{10}x_0^{d-2}+(u_2+\gamma u_0)B_{01}x_0^{d-2}
}
+
\pmat{\xi_0&\xi_{00}x_0^{d-2}\\\eta_0&\eta_{00}x_0^{d-2}}\alpha u_0+\\
+
\pmat{
0&(u_1+\beta u_0)x_1P_1+(u_1+\beta u_0)x_2Q_1+(u_2+\gamma u_0)x_2R_1\\
0&(u_1+\beta u_0)x_1P_2+(u_1+\beta u_0)x_2Q_2+(u_2+\gamma u_0)x_2R_2}=\\
\pmat{
u_1+\beta u_0 &(u_1+\beta u_0)A_{10}x_0^{d-2}+(u_2+\gamma u_0)A_{01}x_0^{d-2}
\\
u_2+\gamma u_0&(u_1+\beta u_0)B_{10}x_0^{d-2}+(u_2+\gamma u_0)B_{01}x_0^{d-2}
}
+
\pmat{\xi_0&\xi_{00}x_0^{d-2}\\\eta_0&\eta_{00}x_0^{d-2}}\alpha u_0+\\
+
\pmat{
0&u_1x_1P_1+u_1x_2Q_1+u_2x_2R_1\\
0&u_1x_1P_2+u_1x_2Q_2+u_2x_2R_2}=\\
\pmat{
u_1&u_1A_{10}x_0^{d-2}+u_2A_{01}x_0^{d-2}+u_1x_1P_1+u_1x_2Q_1+u_2x_2R_1
\\
u_2&u_1B_{10}x_0^{d-2}+u_2B_{01}x_0^{d-2}+u_1x_1P_2+u_1x_2Q_2+u_2x_2R_2
}+\\
\pmat{
\beta+\xi_0\alpha& (\beta A_{10}+\gamma A_{01}+\xi_{00}\alpha)x_0^{d-2}\\
\gamma +\eta_0\alpha&  (\beta B_{10}+\gamma B_{01}+ \eta_{00}\alpha)x_0^{d-2}
}u_0,
\end{split}
\end{equation}
Let us consider the equality $\smat{a&b\\c&d}\cdot \phi^*(\Phi_1)=\Phi_2\cdot \smat{\bar a& \bar b \\0&\bar d}$.

For the entry 1.1 this gives us the equality
\[
a(u_1+\beta u_0)+b(u_2+\gamma u_0)+(a\xi_0+b\eta_0)\alpha u_0=\bar a u_1+\bar a \mu_0 u_0
\]
and hence the comparison of the coefficients yields
\begin{equation}\label{eq:1.1}
a=\bar a,\quad b=0, \quad \beta +\xi_0\alpha = \mu_0.
\end{equation}

For the entry 2.1 this gives us the equality
\[
c(u_1+\beta u_0)+d(u_2+\gamma u_0)+(c\xi_0+d\eta_0)\alpha u_0=\bar a u_2+\bar a \nu_0 u_0
\]
and hence
\begin{equation}\label{eq:2.1}
c=0,\quad d=\bar a,\quad \gamma+\eta_0\alpha=\nu_0.
\end{equation}
Taking into account $b=0$ from~\refeq{eq:1.1} and  restricting the equality for the entry 1.2 to $D_1(p)$ gives
\begin{multline*}
a(u_1A_{10}x_0^{d-2}+u_2A_{01}x_0^{d-2})+a(\beta A_{10}+\gamma A_{01}+\xi_{00}\alpha)x_0^{d-2}u_0=\\
\bar b_{00} u_1+\bar b_{00}\mu_0 u_0+\bar d(
u_1A_{10}x_0^{d-2}+u_2A_{01}x_0^{d-2})
+\bar d \mu_{00}x_0^{d-2}u_0
\end{multline*}
and hence
\begin{equation}\label{eq:1.2}
aA_{10}=\bar b_{00}+\bar d A_{10}, \quad a A_{01}=\bar d A_{01},\quad
a(\beta A_{10}+\gamma A_{01}+\alpha\xi_{00})=\mu_0\bar b_{00}+\bar d\mu_{00}.
\end{equation}
Using $c=0$  from~\refeq{eq:1.2} and restricting the equality for the entry 2.2 to $D_1(p)$ we obtain
\begin{multline*}
d(u_1B_{10}x_0^{d-2}+u_2B_{01}x_0^{d-2})+d(\beta B_{10}+\gamma B_{01}+ \eta_{00}\alpha)x_0^{d-2}u_0=\\
\bar b_{00} u_2+\bar b_{00}\nu_0 u_0+\bar d(
u_1B_{10}x_0^{d-2}+u_2B_{01}x_0^{d-2})
+\bar d \nu_{00}x_0^{d-2}u_0
\end{multline*}
and hence, using $d=a$ from~\refeq{eq:1.2} and~\refeq{eq:1.1}
\begin{equation}\label{eq:2.2}
a B_{10}=\bar d B_{10},\quad aB_{01}=\bar b_{00}+\bar d B_{01},\quad
a(\beta B_{10}+\gamma B_{01}+\alpha \eta_{00})=\bar b_{00}\nu_0+\bar d\nu_{00}.
\end{equation}
From~\refeq{eq:1.1} and~\refeq{eq:2.1} one obtains
$\beta=\mu_0-\alpha\xi_0$ and  $\gamma=\nu_0-\alpha \eta_0$. Then using~\refeq{eq:1.2} we get
\begin{multline*}
\bar d\mu_{00}-a\alpha\xi_{00}=a\beta A_{10}+a\gamma A_{01}-\mu_0\bar b_{00}=
a(\mu_0-\alpha\xi_0)A_{10}+a(\nu_0-\alpha \eta_0)A_{01}-\mu_0A_{10}(a-\bar d)=\\
A_{10}(\bar d\mu_0-a\alpha\xi_0)+aA_{01}\nu_0-a\alpha \eta_0 A_{01}=\\
A_{10}(\bar d\mu_0-a\alpha\xi_0)+\bar d A_{01}\nu_0-a\alpha \eta_0 A_{01}=\\
A_{10}(\bar d\mu_0-a\alpha\xi_0)+A_{01}(\bar d\nu_0-a\alpha \eta_0).
\end{multline*}
Analogously using~\refeq{eq:2.2} we get
\begin{multline*}
\bar d\nu_{00}-a\alpha\eta_{00}=a\beta B_{10}+a\gamma B_{01}-\nu_0\bar b_{00}=\\
a(\mu_0-\alpha\xi_0)B_{10}+a(\nu_0-\alpha \eta_0)B_{01}-\nu_0B_{01}(a-\bar d)=\\
B_{10}(\bar d \mu_0-a\alpha\xi_0)+B_{01}(\bar d\nu_0-a\alpha\eta_0).
\end{multline*}
Therefore, $\bar dB_1-a\alpha B_2$ satisfies~\refeq{eq:tangent equations}, hence $B_1-(\bar d^{-1}a\alpha)\cdot B_2\in T_AX'$. This means that $B_1$ and $B_2$ define the same point in $\P N_A$.

{``$\Leftarrow$''.}
Let now $B_1$ and $B_2$
 be two equivalent normal vectors at $A\in X'$. Without loss of generality we assume $A$ to be as in~\refeq{eq:A special}.
 Let
$\Phi_1=\Phi(A, B_1)$ and $\Phi_2=\Phi(A, B_2)$
be the matrices defining as in~\refeq{eq:Phi(A, B)} the sheaves
 $\cal E_1$ and $\cal E_2$ respectively.
Since $B_1$ and $B_2$ define the same point in $\P N_A$,
it follows that
\[
B_2-\alpha \cdot B_1\in T_A(X_8)
\]
for some $\alpha\in \k^*$.

Let
Let
\[
B_1= \pmat{
\xi_0x_0&\xi_{00}x_0^{d-1} \\
\eta_0x_0& \eta_{00}x_0^{d-1}
}\mod (x_1, x_2), \quad  B_2=
\pmat{
\mu_0x_0&
\mu_{00}x_0^{d-1}\\
\nu_0x_0&
\nu_{00}x_0^{d-1}
}\mod (x_1, x_2).
\]
Take \[
\beta=\mu_0-\xi_0\alpha,\quad \gamma=\nu_0-\eta_0 \alpha,
\]
and let
\begin{equation}\label{eq:L automorphism}
\phi_1=\smat{\alpha&\beta&\gamma\\
0&1&0\\
0&0&1
}:\P_2\ra \P_2,\quad \point{u_0, u_1, u_2}\mto \langle(u_0, u_1, u_2)
\smat{\alpha&\beta&\gamma\\
0&1&0\\
0&0&1
}\rangle.
\end{equation}
Note that the automorphisms of the form~\refeq{eq:L automorphism} are exactly the automorphisms of $D_1\iso \P_2$ acting identically on $L$.

Consider
\(
\phi:D(p)\ra D(p)
\)
such that $\phi|_{D_1}=\phi_1$ and $\phi|_{D_0}=\id_{D_0}$.
Using the tangent equations~\refeq{eq:tangent equations} and that $u_0x_1=u_0x_2=0$ one checks that  $\phi^*(\Phi_1)=\Phi_2$.
Indeed, by~\refeq{eq:pull back matrix}
\begin{multline*}
\phi^*(\Phi_1)=
\pmat{
u_1&u_1A_{10}x_0^{d-2}+u_2A_{01}x_0^{d-2}+u_1x_1P_1+u_1x_2Q_1+u_2x_2R_1\\
u_2&u_1B_{10}x_0^{d-2}+u_2B_{01}x_0^{d-2}+u_1x_1P_2+u_1x_2Q_2+u_2x_2R_2}+
\\
\pmat{
\beta +\xi_0\alpha &(\xi_{00}\alpha+\beta A_{10}+\gamma A_{01})x_0^{d-2}\\
\gamma +\eta_0\alpha &(\eta_{00}\alpha +\beta B_{10}+\gamma B_{01})x_0^{d-2}} u_0=
\\
\pmat{
u_1&u_1A_{10}x_0^{d-2}+u_2A_{01}x_0^{d-2}+u_1x_1P_1+u_1x_2Q_1+u_2x_2R_1\\
u_2&u_1B_{10}x_0^{d-2}+u_2B_{01}x_0^{d-2}+u_1x_1P_2+u_1x_2Q_2+u_2x_2R_2}+
\pmat{\mu_0&\mu_{00}x_0^{d-2}\\\nu_0&\nu_{00}x_0^{d-2}} u_0= \Phi_2.
\end{multline*}
Therefore, there is an
isomorphism $\phi^*(\cal E_1)\iso \cal E_2$,
which means that the sheaves $\cal E_1$ and $\cal E_2$ are equivalent.
\end{proof}

This proposition immediately implies the main statement of this note.

\begin{tr}\label{tr: main}
Let $\tilde M=\Bl_{M'}(M)$. Then the exceptional divisor of this blow up consists of the equivalence classes of $R$-bundles.
\end{tr}
%\mymargin{II.10 fixed}

%\bibliography{literature}

\begin{thebibliography}{10}

\bibitem{BeilinsonEng}
A.A. Beilinson.
\newblock {Coherent sheaves on ${\mathbb{P}}^n$ and problems of linear
  algebra.}
\newblock {\em Funct. Anal. Appl.}, 12:214--216, 1979.

\bibitem{DrezetMaican4m}
Jean-Marc Dr{\'e}zet and Mario Maican.
\newblock On the geometry of the moduli spaces of semi-stable sheaves supported
  on plane quartics.
\newblock {\em Geom. Dedicata}, 152:17--49, 2011.

\bibitem{Freiermuth-Diplom}
Hans~Georg Freiermuth.
\newblock On the moduli space ${M}\sb {P}(\mathbb {P}\sb 3)$ of semi-stable
  sheaves on $\mathbb{P}\sb 3$ with {H}ilbert polynomial ${P}(m)=3m+1$.
\newblock Diplomarbeit, Universut\"at Kaiserslautern, Germany, December 2000.

\bibitem{Freiermuth}
Hans~Georg Freiermuth and G{\"u}nther Trautmann.
\newblock On the moduli scheme of stable sheaves supported on cubic space
  curves.
\newblock {\em Amer. J. Math.}, 126(2):363--393, 2004.

\bibitem{EGA-IV-II}
A.~Grothendieck.
\newblock \'{E}l\'ements de g\'eom\'etrie alg\'ebrique. {IV}. \'{E}tude locale
  des sch\'emas et des morphismes de sch\'emas. {II}.
\newblock {\em Inst. Hautes \'Etudes Sci. Publ. Math.}, (24):231, 1965.

\bibitem{IenaTrmPaper1}
O.~{Iena} and G.~{Trautmann}.
\newblock {Modification of the Simpson moduli space $M_{3m+1}(\mathbb P_2)$ by
  vector bundles (I)}.
\newblock {\em ArXiv e-prints}, December 2010.

\bibitem{MyGermanDiss}
Oleksandr {Iena}.
\newblock {\em Modification of Simpson moduli spaces of $1$-dimensional sheaves
  by vector bundles, an experimental example}.
\newblock {PhD} thesis in {M}athematics, Technische Universit\"at
  Kaiserslautern, Germany, 2009.

\bibitem{LePotier}
J.~Le~Potier.
\newblock Faisceaux semi-stables de dimension {$1$} sur le plan projectif.
\newblock {\em Rev. Roumaine Math. Pures Appl.}, 38(7-8):635--678, 1993.

\bibitem{Maican6m+1}
M.~{Maican}.
\newblock {On the moduli space of semi-stable plane sheaves with Euler
  characteristic one and supported on sextic curves}.
\newblock {\em ArXiv e-prints}, April 2011.

\bibitem{MaicanHilb}
M.~{Maican}.
\newblock {Relative Hilbert schemes and moduli spaces of torsion plane
  sheaves}.
\newblock {\em ArXiv e-prints}, July 2013.

\bibitem{MaicanTwoSemiSt}
Mario Maican.
\newblock {On two notions of semistability.}
\newblock {\em Pac. J. Math.}, 234(1):69--135, 2008.

\bibitem{MaicanDuality}
Mario Maican.
\newblock {A duality result for moduli spaces of semistable sheaves supported
  on projective curves.}
\newblock {\em Rend. Semin. Mat. Univ. Padova}, 123:55--68, 2010.

\bibitem{TrautmannTikhMark}
Dimitri Markushevich, Alexander~S. Tikhomirov, and G{\"u}nther Trautmann.
\newblock Bubble tree compactification of moduli spaces of vector bundles on
  surfaces.
\newblock {\em Cent. Eur. J. Math.}, 10(4):1331--1355, 2012.

\bibitem{Okonek-Schneider-Spindler}
Christian Okonek, Michael Schneider, and Heinz Spindler.
\newblock {\em Vector bundles on complex projective spaces}, volume~3 of {\em
  Progress in Mathematics}.
\newblock Birkh\"auser Boston, Mass., 1980.

\bibitem{Simpson1}
Carlos~T. Simpson.
\newblock Moduli of representations of the fundamental group of a smooth
  projective variety. {I}.
\newblock {\em Inst. Hautes \'Etudes Sci. Publ. Math.}, (79):47--129, 1994.

\end{thebibliography}
%\bibliographystyle{plain}
%\end{document}

\def\cprime{$'$} \def\cprime{$'$}

\end{document}